\documentclass[12pt,a4paper]{article}

\usepackage{titlesec}

\titlespacing*{\section}
{0pt}{5.5ex plus 1ex minus .2ex}{4.3ex plus .2ex}
\titlespacing*{\subsection}
{0pt}{5.5ex plus 1ex minus .2ex}{4.3ex plus .2ex}
\usepackage{amsmath,  amsthm, amssymb, color}
\usepackage{enumerate, hyperref}
\usepackage{graphicx}
\usepackage{tikz}
\usetikzlibrary{bbox}

\usepackage[open,openlevel=2,atend]{bookmark}

\usepackage[abbrev,msc-links,nobysame]{amsrefs}  
\usepackage{doi}

\renewcommand{\PrintDOI}[1]{\doi{#1}}

\usepackage{amsmath}

\usepackage{pifont}
 \usepackage{amsthm}
\usepackage{amsgen,amstext,amsbsy,amsopn}
\usepackage{float}
\usepackage{enumitem}
\usepackage{xcolor}

\usetikzlibrary{intersections,calc,arrows.meta}

\newtheorem{thm}{Theorem}[section]

\newtheorem{prop}{Proposition}[section]
\newtheorem{lem}{Lemma}[section]
\newtheorem{cor}{Corollary}

\newtheorem{prob}{Problem}[section]

\newtheorem{claim}{Claim}[section]
\newtheorem{case}{Case}
\newtheorem{remark}{Remark}[section]

\usepackage{hyperref}
\hypersetup{
    colorlinks=true,
    linkcolor=blue,
    citecolor=blue,
    urlcolor=blue
}

\usetikzlibrary{backgrounds}
\usetikzlibrary{arrows}
\usetikzlibrary{shapes,shapes.geometric,shapes.misc}
\pgfdeclarelayer{edgelayer}
\pgfdeclarelayer{nodelayer}
\pgfsetlayers{background,edgelayer,nodelayer,main}
\tikzstyle{none}=[inner sep=0mm]
\tikzstyle{blacknode}=[fill=black, draw=black, shape=circle, minimum
size=0.20cm, inner sep=0pt]
\tikzstyle{whitenode}=[fill={rgb,255: red,245; green,245; blue,245},
draw=black, shape=circle, minimum size=0.20cm, inner sep=1pt]
\tikzstyle{whitenode_v1}=[fill={rgb,255: red,245; green,245; blue,245},
draw=black, shape=circle, minimum size=0.4cm, inner sep=0.8pt, scale=1]
\tikzstyle{blacknode_v1}=[fill=black, draw=black, shape=circle, minimum
size=0.1cm, inner sep=0pt]
\tikzstyle{black_bold}=[-, draw=black, line width=0.6mm]
\tikzstyle{blackedge}=[-, draw=black, fill=none, line width=0.3mm]
\tikzstyle{rededge}=[-, line width=0.3mm, draw=red]
\tikzstyle{blueedge}=[-, line width=0.3mm, draw=blue]
\tikzstyle{grayedge}=[-,line width=0.2mm, draw={rgb,255: red,64; green,64; blue,64}]

\tikzstyle{shadow_silver}=[-, draw=black, fill={rgb,255: red,186; green,186; blue,186}, line width=0.45mm]

\definecolor{xwhite}{RGB}{245,245,245}%
\definecolor{xback}{RGB}{0,0,0}%
\definecolor{xblue}{RGB}{0,0,139}%
\definecolor{xgreen}{RGB}{50,205,50}%
\definecolor{xcrimson}{RGB}{220,20,60}%
\definecolor{xgold}{RGB}{255,215,0}
\definecolor{xmoccasin}{RGB}{255,228,181}

\DeclareMathOperator{\odd}{odd}

\newcommand{\yes}{\ding{51}}
\newcommand{\no}{\ding{55}}
\usepackage{lineno} 

\begin{document}

\title{\textbf{\Large Near-perfect matchings in highly connected 1-planar graphs with a local crossing constraint}}

\author{
 \normalsize  Licheng Zhang \thanks{E-mail: \texttt{lczhangmath@163.com}. College of Mathematics and Statistics, Hunan  Normal  University} \and
  \normalsize  Yuanqiu Huang \thanks{Corresponding author. E-mail: \texttt{hyqq@hunnu.edu.cn}. College of Mathematics and Statistics, Hunan  Normal  University} \and
  \normalsize Zhangdong Ouyang \thanks{ E-mail: \texttt{oymath@163.com}. College of Mathematics and Statistics, Hunan First Normal   University }
}

\date{}
\maketitle

\noindent
\begin{abstract}
 For planar graphs, it is well known that high connectivity implies a Hamiltonian cycle and hence any 4-connected planar graph has a near-perfect matching.   Nevertheless, whether 6-connected 1-planar graphs admit near-perfect matchings remains largely open. The prior art established this for 4-connected 1-planar graphs only when each crossing involves four endpoints that induce a $K_4$.
In this paper, we study  6-connected 1-planar graphs that are drawn such that at all crossings the four endpoints induce a 4-cycle (plus perhaps more edges). We show that these have a near-perfect matching, and in fact even stronger, their scattering number is at most one. Moreover, under the local crossing restriction, the requirement of 6-connectivity is best possible; this is witnessed by explicit constructions due to Biedl and Fabrici et al.

\end{abstract}
\noindent
\textbf{Keywords. Connectivity, Matching numbers,   1-planar graphs }

\section{Introduction} Throughout this paper, we only consider \emph{simple graphs} (i.e., graphs without loops or multiple edges), unless otherwise stated. Let $G=(V(G),E(G))$ be a graph, where $V(G)$ and $E(G)$ represent the vertex set and edge set of $G$, respectively. Let $n(G) = |V(G)|$ and $e(G) = |E(G)|$ denote the \emph{order} and \emph{size} of $G$, respectively. Two distinct edges  with a common endvertex are \emph{adjacent}. 
 A \emph{matching} of a graph $G$ is a set of pairwise non-adjacent edges, that is, no two edges in the matching share a common vertex. The matching number $\alpha^{\prime}(G)$ of a graph $G$ is the size of a maximum matching in $G$. 
 Usually, a perfect matching (also called a 1-factor) of a graph $G$ of order $n$ is defined as a matching of size $\tfrac{n}{2}$, while a near-perfect matching is a matching of size $\tfrac{n-1}{2}$. Clearly, it is possible for $G$ to have a perfect matching only if $n$ is even, and a near-perfect matching only if $n$ is odd.  For simplicity, in this paper, we adopt a unified definition: a matching $M$ of a graph of order $n$ is called a \emph{ near-perfect matching} if $|M|=\lfloor \tfrac{n}{2} \rfloor$. A graph is \emph{Hamiltonian} if it contains a cycle visiting every vertex exactly once (such a cycle is called a \emph{Hamiltonian cycle}). Clearly,  in a Hamiltonian graph $G$ of order $n$, one can find a matching of size $\lfloor n/2 \rfloor$ in a Hamiltonian cycle, yielding a near-perfect matching. 

Matchings in graphs have been extensively studied. In particular, the existence of perfect matchings  or the identification of a large matching in a graph has long attracted attention. One of the earliest results is Petersen's matching theorem from 1891, stating that every bridgeless cubic graph contains a perfect matching \cite{MR1554815}.
Another well-known result is Tutte's 1-factor theorem, which gives a necessary and sufficient condition for the existence of a perfect matching \cite{MR0023048}. From an algorithmic perspective, maximum matchings of  general graphs $G$ can be efficiently constructed using Edmonds' blossom algorithm in $O(e(G)n(G)^2)$ time \cite{MR0177907}.
 For more results on (near-perfect) matchings of graphs, we refer the reader to the classical book of Lov\'{a}sz and Plummer \cite{MR0859549}, or to a more recent book of Lucchesi and Murty \cite{MR4769470}.  

In this paper, we focus on matchings in graphs with specific drawings.  Notably, the classical theorem of Tutte \cite{MR0081471} states that every 4-connected planar graph is Hamiltonian. Thus any 4-connected planar graph  has a near-perfect matching. Later, Nishizeki and Baybars \cite{MR0548625} gave lower bounds on matching number  for planar graphs as a function of the minimum degree and the connectivity. To be more specific, for example, every 3-connected planar graph with minimum degree 3 and order $n$ has a matching of size at least $\frac{n+4}{3}$ if $n\ge 14$.
In recent years, there has been growing interest in generalizations of planar graphs that allow some edge crossings, with particular attention to \emph{1-planar} graphs, which can be drawn in the plane so that each edge is crossed at most once. A graph together with a 1-planar drawing is a \emph{1-plane} graph. For instance, the 2017 survey on 1-planar graphs by Kobourov, Liotta, and Montecchiani \cite[p.~65]{MR3697129} mentioned as an open problem the development of matching results for 1-planar graphs analogous to those of Nishizeki and Baybars for planar graphs. In 2021, Biedl and Wittnebel \cite{MR4429150} investigated 1-planar graphs with minimum degree $\delta$ (for $\delta = 3, \ldots, 7$), providing lower bounds on the size of a matching; the bounds are tight for $\delta = 3, 4, 5$.  For $\delta=6$, the bounds were later improved (and made tight) by the last two authors in this paper and Dong \cite{MR4504132}.  
A graph $G$ from a family of graphs $\mathcal{G}$ is \emph{maximal} if joining any pair of nonadjacent vertices by an edge results in a graph that is not in $\mathcal{G}$.
 Recently, Biedl and Wittnebel \cite{MR4813988}   investigated the size of large matchings in maximal 1-planar graphs.  In particular they show that every 3-connected maximal 1-planar graph has a matching of size at least $\frac{2 n+6}{5}$; the bound decreases to $\frac{3 n+14}{10}$ if the graph need not be 3-connected. 
There are 1-planar graphs with connectivity 5 that have no near-perfect matchings \cite{Biedl}. 
Of particular interest is the following problem posed by Biedl and Wittnebel \cite{MR4429150} concerning matchings in highly connected 1-planar graphs. In fact,  the same problem was previously proposed by Fujisawa, Segawa and Suzuki \cite{MR3846896} for 5-connected 1-planar graphs.
\begin{prob}[Biedl and Wittnebel, \cite{MR4429150}]\label{prob:0}
Does every 6-connected (or 7-connected)  1-planar graph have a  near-perfect matching?
\end{prob}

 It is worth noting that  any 1-planar graph of order $n$ where $n \geq 3$ has size at most $4 n-8$  \cite{MR1606052}. Consequently, any 1-planar graph has a minimum degree of at most 7, and its connectivity is therefore also at most 7.  Meanwhile, the hamiltonicity of 1-planar graphs has also been gradually investigated.
A 1-planar graph of order $n$  is \emph{optimal} if it has exactly $4 n-8$ edges.  Hud\'{a}k, Madaras, and Suzuki \cite{MR2993519} showed that every optimal 1-planar graph, as well as every 7-connected maximal 1-planar graph, is Hamiltonian. Moreover, the hamiltonicity result for optimal 1-planar graphs (it is known that every optimal 1-planar graph is 4-connected and maximal \cite{MR2746706}) and 7-connected maximal 1-planar graphs was later extended to 4-connected maximal  1-planar  graphs \cite{MR4130388}. As a result, every 4-connected maximal 1-planar graph admits a near-perfect matching. In fact, more refined results were established in \cite{MR4130388}. To simplify notation and allow broader application, we revise some original terminology from \cite{MR4130388}. 
Let $D$ be a 1-planar drawing of a graph $G$, and let $ab$ and $cd$ be two edges that cross at a crossing $\alpha$ in $D$. The edges in $G[\{a, b, c, d\}] \backslash\{a b, c d\}$ are called the \emph{associated edges} of the crossing pair $a b$ and $c d$  or of the crossing $\alpha$. The crossing $\alpha$ of two edges $ab$ and $cd$ in a 1-planar drawing $D$ is called \emph{type-$k$} if the number of its associated edges is at least 
$k$.  Clearly,  $0\le k\le 4$. Furthermore, if a 1-planar graph $G$ admits a 1-planar drawing $D$ in which all crossings are type-$k$,  $D$ is called a \emph{type-$k$ 1-planar} drawing, and  $G$ is called a \emph{type-$k$ 1-planar} graph.  Clearly, any type-$k$ 1-planar graph is  type-$(k-1)$.    The crossing $\alpha$ of two edges $ab$ and $cd$ in a 1-planar drawing $D$ is called \emph{type-2A} if it has exactly two associated edges and the two edges form a matching. A type-2 1-plane graph $G$ is \emph{type-2A} if every crossing in $G$ with exactly two associated edges is of type-2A. For illustration, Figure \ref{fig:a} serves as an example.  Among the crossings in the figure, $\alpha_1$ is type-0, $\alpha_6$ is type-1, $\alpha_4$ is   type-3, and $\alpha_5$ is type-4. Regarding $\alpha_2$ and $\alpha_3$, they are  type-2. Specifically, $\alpha_2$ is  type-2A, while $\alpha_3$ is not type-2A. 
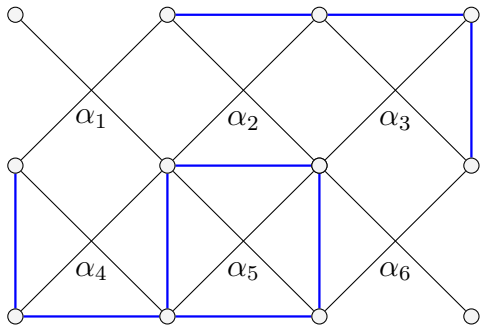
\begin{figure}[H]
\centering\begin{tikzpicture}[scale=0.5]
	\begin{pgfonlayer}{nodelayer}
		\node [style=whitenode] (0) at (-7, 2) {};
		\node [style=whitenode] (1) at (-3, 2) {};
		\node [style=whitenode] (2) at (-7, -2) {};
		\node [style=none] (6) at (-5, -0.75) {$\alpha_1$};
		\node [style=whitenode] (8) at (-3, -2) {};
		\node [style=whitenode] (9) at (1, 2) {};
		\node [style=whitenode] (10) at (1, -2) {};
		\node [style=whitenode] (11) at (1, -2) {};
		\node [style=whitenode] (12) at (-7, -6) {};
		\node [style=whitenode] (13) at (-3, -6) {};
		\node [style=whitenode] (14) at (1, -6) {};
		\node [style=none] (17) at (-1, -0.8) {$\alpha_2$};
		\node [style=none] (18) at (-5, -4.8) {$\alpha_4$};
		\node [style=none] (19) at (-1, -4.8) {$\alpha_5$};
		\node [style=whitenode] (20) at (5, 2) {};
		\node [style=whitenode] (23) at (5, -2) {};
		\node [style=whitenode] (24) at (5, -6) {};
		\node [style=none] (25) at (3, -0.8) {$\alpha_3$};
		\node [style=none] (26) at (3, -4.8) {$\alpha_6$};
	\end{pgfonlayer}
	\begin{pgfonlayer}{edgelayer}
		\draw (1) to (2);
		\draw (0) to (8);
		\draw (1) to (11);
		\draw (9) to (8);
		\draw [style=blueedge] (1) to (9);
		\draw (8) to (12);
		\draw (2) to (13);
		\draw [style=blueedge] (12) to (13);
		\draw (8) to (14);
		\draw (11) to (13);
		\draw [style=blueedge] (11) to (14);
		\draw [style=blueedge] (14) to (13);
		\draw  (20) to (11);
		\draw (9) to (23);
		\draw [style=blueedge] (8) to (11);
		\draw [style=blueedge] (9) to (20);
		\draw [style=blueedge] (20) to (23);
		\draw [style=blueedge] (8) to (13);
		\draw [style=blueedge] (2) to (12);
		\draw (23) to (14);
		\draw (11) to (24);
	\end{pgfonlayer}
\end{tikzpicture}
\caption{A 1-plane graph containing crossings of various types.}
\label{fig:a}
\end{figure}

Note that type-4 1-planar graphs correspond to locally maximal 1-planar graphs, while type-3 1-planar graphs correspond to weakly locally maximal 1-planar graphs, as discussed in \cite{MR4130388}.
Fabrici, Harant, Madaras, Mohr, Sot\'{a}k, and Zamfirescu \cite{MR4130388} showed that if a 4-connected graph admits a type-3 1-planar drawing in which at most three crossing pairs have exactly three associated edges, then the graph is Hamiltonian.
Moreover, Fabrici et al. \cite{MR4130388} constructed a class of 5-connected type-3 1-plane graphs with shortness exponent at most $\frac{\log 20}{\log 21}$; such a bound also implies that these graphs are non-Hamiltonian (see Theorem 2 in their paper for details).
Furthermore, Fabrici et al. \cite{MR4130388} left an open problem asking whether every 6-connected type-3 1-planar graph is Hamiltonian; the authors also proposed the more general problem asking whether every 6-connected 1-planar graph is Hamiltonian.  It should be noted that Hudák, Madaras, and Suzuki raised the same question in 2012 \cite{MR2993519}.

Within this line of research, once  Hamiltonian cycles or perfect matchings are known to exist in some classes of 1-planar graphs, such as optimal and 7-connected maximal 1-planar graphs, it is natural to investigate stronger properties, including matching extendability \cite{MR3846896, MR4606109,MR4715705}.
At a fundamental level, the existence of Hamiltonian cycles and near-perfect matchings thus constitutes one of the central problems in the study of 1-planar graphs.
In this paper, we go a step further by establishing a result on the existence of near-perfect matchings under a relaxed local crossing condition.

\begin{thm}\label{thm:1}
Let $G$ be a type-2A 1-planar graph. If $G$ is 6-connected, then $G$ has a near-perfect matching.
\end{thm}

The following corollary is an immediate consequence of the above theorem.
\begin{cor}\label{cor:1}
Every 6-connected type-3 1-planar graph has a near-perfect matching.
\end{cor}

At present, researchers attempt to address Problem \ref{prob:0} (or Hamiltonicity) by imposing additional conditions beyond connectivity and 1-planarity, such as maximality \cite{Biedl}, type-4 \cite{MR4130388}, optimality \cite{MR2993519}, or chordality \cite{MR4933418}. Interestingly, under these conditions, the existence of near-perfect matchings or even Hamiltonicity can often be established without requiring 6-connectivity, with 4-connectedness typically sufficing. In contrast to these known results, 6-connectivity plays a critical role when the additional type-2A (or even type-3) condition is imposed. This is because there exist infinitely many type-3 1-planar graphs $G$ with connectivity 5 that admit no near-perfect matching. We follow the graph construction of Fabrici et al. \cite{MR4130388}, although their focus was on the upper bound of the circumference (i.e., the length of the longest cycle) of there graphs. Firstly, Fabrici et al. constructed a type-3 1-planar graph $G_0$ by taking the structure $H$ shown in Figure~\ref{fig:H}, adding a new white vertex, and joining its five half-edges to this vertex.  It is easy to check that $G_0$ is 5-connected.  Moreover, since $G_0$ contains 20 black vertices and 22 white vertices, and the set of white vertices is independent, it follows from Tutte-Berge  formula (see Lemma \ref{lem:tutteberge} in the next section) that $\alpha'(G_0)<\lfloor \frac{n(G_0)}{2}\rfloor$. Now, $G_0$ can be generalized to a class of type-3 1-planar graphs with connectivity 5. A feasible strategy is to adjust the two 20-vertex cycles in $H$ (Figure \ref{fig:H}, red and blue in the electronic version), extending them to length $4 k$ where $k\ge 5$ and adding specific edges to ensure the resulting graph meets the required properties.  We omit the technical details here for simplicity. Moreover, Biedl  \cite{Biedl}   also constructed 5-connected 1-planar graphs without near-perfect matchings. A direct verification shows that these 1-planar graphs are also type-3.

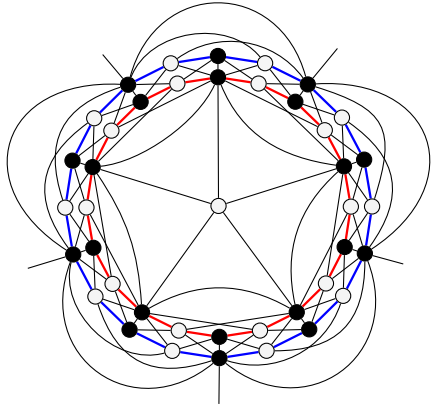
\begin{figure}
\centering
\begin{tikzpicture}[scale=0.4, rotate=8]
	\begin{pgfonlayer}{nodelayer}
		\node [style=whitenode] (1) at (2.39728, -2.72433) {};
		\node [style=blacknode] (2) at (0.82663, -2.75) {};
		\node [style=whitenode] (3) at (-0.676, -2.29541) {};
		\node [style=blacknode] (4) at (-1.48316, -0.89266) {};
		\node [style=whitenode] (5) at (-2.3111, 0.17674) {};
		\node [style=blacknode] (6) at (-2.7698, 1.45021) {};
		\node [style=whitenode] (7) at (-2.8089, 2.80427) {};
		\node [style=blacknode] (8) at (-2.4205, 4.10265) {};
		\node [style=whitenode] (9) at (-2.1454, 5.69783) {};
		\node [style=blacknode] (10) at (-0.55795, 6.02281) {};
		\node [style=whitenode] (11) at (0.72387, 6.45637) {};
		\node [style=blacknode] (12) at (2.07581, 6.47612) {};
		\node [style=whitenode] (13) at (3.67608, 6.7217) {};
		\node [style=blacknode] (14) at (4.96904, 5.83081) {};
		\node [style=whitenode] (15) at (5.92828, 4.58598) {};
		\node [style=blacknode] (16) at (6.45759, 3.10424) {};
		\node [style=whitenode] (17) at (6.5, 1.52964) {};
		\node [style=blacknode] (18) at (6.04753, 0.01992) {};
		\node [style=whitenode] (19) at (5.14512, -1.27193) {};
		\node [style=blacknode] (20) at (3.88565, -2.21687) {};
		\node [style=whitenode] (21) at (2.27598, -2.03085) {};
		\node [style=blacknode] (22) at (0.92399, -2.05071) {};
		\node [style=whitenode] (23) at (-0.36947, -1.65878) {};
		\node [style=blacknode] (24) at (-1.9692, -1.40502) {};
		\node [style=whitenode] (25) at (-2.9283, -0.16015) {};
		\node [style=blacknode] (26) at (-3.4576, 1.32158) {};
		\node [style=whitenode] (27) at (-3.5, 2.89616) {};
		\node [style=blacknode] (28) at (-3.0476, 4.40588) {};
		\node [style=whitenode] (29) at (-1.64259, 5.21236) {};
		\node [style=blacknode] (30) at (-0.88564, 6.64262) {};
		\node [style=whitenode] (31) at (0.60276, 7.1504) {};
		\node [style=blacknode] (32) at (2.17345, 7.1764) {};
		\node [style=whitenode] (33) at (3.36961, 6.08511) {};
		\node [style=blacknode] (34) at (4.48346, 5.31927) {};
		\node [style=whitenode] (35) at (5.3111, 4.24959) {};
		\node [style=blacknode] (36) at (5.76922, 2.97592) {};
		\node [style=whitenode] (37) at (5.80782, 1.62191) {};
		\node [style=blacknode] (38) at (5.4193, 0.3237) {};
		\node [style=whitenode] (39) at (4.64234, -0.78642) {};
		\node [style=blacknode] (40) at (3.5575, -1.59708) {};
		\node [style=none] (84) at (-1.575, 7.75) {};
		\node [style=none] (85) at (6.1, 6.875) {};
		\node [style=none] (86) at (7.25, -0.5) {};
		\node [style=none] (87) at (0.6, -4.25) {};
		\node [style=none] (88) at (-5, 1.075) {};
		\node [style=whitenode] (89) at (1.5, 2.25) {};
		\node [style=none] (90) at (2.75, 7.25) {};
		\node [style=none] (91) at (2.5, 5.75) {};
		\node [style=none] (92) at (0.75, 7.5) {};
		\node [style=none] (93) at (1, 5.75) {};

	\end{pgfonlayer}
	\begin{pgfonlayer}{edgelayer}
		\draw (1) to (22);
		\draw (1) to (40);
		\draw (2) to (21);
		\draw (2) to (22);
		\draw (2) to (23);
		\draw (3) to (4);
		\draw (3) to (22);
		\draw (4) to (24);
		\draw (4) to (25);
		\draw (5) to (24);
		\draw (5) to (26);
		\draw (6) to (25);
		\draw (6) to (26);
		\draw (6) to (27);
		\draw (7) to (26);
		\draw (7) to (28);
		\draw (8) to (9);
		\draw (8) to (27);
		\draw (8) to (28);
		\draw (9) to (10);
		\draw (10) to (30);
		\draw (10) to (31);
		\draw (11) to (30);
		\draw (11) to (32);
		\draw (12) to (13);
		\draw (12) to (31);
		\draw (12) to (32);
		\draw (13) to (34);
		\draw (14) to (33);
		\draw (14) to (34);
		\draw (14) to (35);
		\draw (15) to (34);
		\draw (15) to (36);
		\draw (16) to (35);
		\draw (16) to (36);
		\draw (16) to (37);
		\draw (17) to (36);
		\draw (17) to (38);
		\draw (18) to (37);
		\draw (18) to (38);
		\draw (18) to (39);
		\draw (19) to (38);
		\draw (19) to (40);
		\draw (20) to (21);
		\draw (20) to (39);
		\draw (20) to (40);
		\draw (23) to (24);
		\draw (28) to (29);
		\draw (29) to (30);
		\draw (32) to (33);
		\draw (33) to (36);
		\draw (35) to (12);
		\draw (12) to (29);
		\draw (11) to (8);
		\draw [bend right, looseness=0.75] (8) to (12);
		\draw [in=165, out=-75] (12) to (36);
		\draw (36) to (39);
		\draw (40) to (37);
		\draw [bend right] (36) to (40);
		\draw (40) to (23);
		\draw (4) to (21);
		\draw [bend left] (4) to (40);
		\draw (8) to (5);
		\draw (7) to (4);
		\draw [bend left=15] (8) to (4);
		\draw [bend left=75, looseness=1.50] (30) to (14);
		\draw [bend left=60] (30) to (13);
		\draw [bend right=45] (14) to (31);
		\draw [bend left=285, looseness=1.75] (30) to (26);
		\draw [bend left=45, looseness=1.25] (26) to (9);
		\draw [bend right=45, looseness=1.25] (30) to (27);
		\draw [bend right=75, looseness=1.50] (26) to (2);
		\draw [bend right=60] (26) to (3);
		\draw [bend right=315, looseness=1.25] (2) to (25);
		\draw [bend right=75, looseness=1.50] (2) to (18);
		\draw [bend right=45] (2) to (19);
		\draw [bend left=45, looseness=1.25] (18) to (1);
		\draw [bend right=75, looseness=1.75] (18) to (14);
		\draw [bend right=45, looseness=1.25] (18) to (15);
		\draw [bend left=45, looseness=1.50] (14) to (17);
		\draw  (84.center) to (30);
		\draw  (14) to (85.center);
		\draw  (86.center) to (18);
		\draw  (2) to (87.center);
		\draw  (26) to (88.center);
		\draw (8) to (89);
		\draw (89) to (4);
		\draw (89) to (40);
		\draw (89) to (36);
		\draw (89) to (12);
		\draw [style=rededge] (10) to (29);
		\draw [style=rededge] (29) to (8);
		\draw [style=rededge] (8) to (7);
		\draw [style=rededge] (7) to (6);
		\draw [style=rededge] (6) to (5);
		\draw [style=rededge] (5) to (4);
		\draw [style=rededge] (4) to (23);
		\draw [style=rededge] (23) to (22);
		\draw [style=rededge] (22) to (21);
		\draw [style=rededge] (21) to (40);
		\draw [style=rededge] (40) to (39);
		\draw [style=rededge] (39) to (38);
		\draw [style=rededge] (38) to (37);
		\draw [style=rededge] (37) to (36);
		\draw [style=rededge] (36) to (35);
		\draw [style=rededge] (35) to (34);
		\draw [style=rededge] (34) to (33);
		\draw [style=rededge] (33) to (12);
		\draw [style=rededge] (12) to (11);
		\draw [style=rededge] (11) to (10);
		\draw [style=blueedge] (31) to (32);
		\draw [style=blueedge] (32) to (13);
		\draw [style=blueedge] (13) to (14);
		\draw [style=blueedge] (14) to (15);
		\draw [style=blueedge] (15) to (16);
		\draw [style=blueedge] (16) to (17);
		\draw [style=blueedge] (17) to (18);
		\draw [style=blueedge] (18) to (19);
		\draw [style=blueedge] (19) to (20);
		\draw [style=blueedge] (20) to (1);
		\draw [style=blueedge] (1) to (2);
		\draw [style=blueedge] (2) to (3);
		\draw [style=blueedge] (3) to (24);
		\draw [style=blueedge] (24) to (25);
		\draw [style=blueedge] (25) to (26);
		\draw [style=blueedge] (26) to (27);
		\draw [style=blueedge] (27) to (28);
		\draw [style=blueedge] (28) to (9);
		\draw [style=blueedge] (9) to (30);
		\draw [style=blueedge] (30) to (31);
	\end{pgfonlayer}
\end{tikzpicture}

\caption{The structure $H$ (with the red and blue cycles highlighted in the electronic version)}
\label{fig:H}
\end{figure}


\noindent \textbf{Outline of the paper.}
The rest of the paper is structured as follows. In the rest of the section,  we introduce notations and terminologies. In Section 2, we present preliminary lemmas, focusing mainly on the properties of type-2A 1-planar graphs. In Section 3, we show  Theorem \ref{thm:1}.

\noindent \textbf{Notations and terminologies.}  
Let $G$ be a graph.
Let $v\in V(G)$ and $S\subseteq V(G)$. 
Then  $N_G(v)$   denotes the set of neighbors
of $v$ in $G$.  For $S \subseteq V(G)$, let $N_G(S)=\left(\bigcup_{x \in S} N_G(x)\right) \backslash S$.
The subgraphs of $G$ induced on $S$ and $V(G)\setminus S$ are denoted by
 $G[S]$ and $G-S$, respectively.
Let $V_1,
V_2\subseteq V(G)$ be two disjoint vertex sets. Then $N_G(V_1, V_2) = \{ v \in V_2 \mid \text{$v$ is adjacent to some vertex in $V_1$} \}$.
 Then $E_G(V_1,V_2)$ is the set
of edges in $G$  with one endvertex in $V_1$ and the other endvertex in $V_2$ and  $e_G(V_1,V_2):=|E_G(V_1,V_2)|$. 
Let $X \subseteq E(G)$ be a set of edges. We denote by $G - X$ the graph obtained from $G$ by removing all edges in $X$, i.e.,
$
G - X := (V(G), E(G) \setminus X).
$
Let $G-V_0$ denote the subgraph of $G$ induced by $V(G) \backslash V_0$ when $V_0 \neq V(G)$. A subset $S$ of $V(G)$ is called a \emph{vertex-cut} if $G-S$ is disconnected, and a \emph{vertex-cut} $S$ is called a $k$-vertex-cut if $|S|=k$. For a noncomplete graph $G$, its \emph{connectivity}, denoted by $\kappa(G)$, is defined to be the minimum value of $|S|$ over all vertex-cuts $S$ of $G$, and $\kappa(G)=n(G)-1$ if $G$ is complete and $n(G)\ge 2$. A graph $G$ is \emph{$k$-connected} if $\kappa(G) \geq k$. 

A \emph{component} of a graph $G$ is a maximal connected subgraph of $G$. An \emph{odd component} is a component that contains an odd number of vertices. For a subset $S \subseteq V(G)$, denote by $c(G - S)$ the number of components in the graph $G - S$, and by $\odd(G - S)$ the number of odd components in $G - S$.

A \emph{drawing} of a graph $G$ is a mapping that assigns to each vertex of $G$ a distinct point
in the plane and to each edge a continuous arc connecting its endpoints. A drawing is  \emph{good}
if no edge crosses itself, no two edges cross more than once, and no two edges incident with the same vertex cross each other. In this paper, all drawings are assumed to be good. In a graph drawing, an edge is \emph{crossed} if it intersects another edge at a crossing (i.e., a point of intersection other than their common endpoints); otherwise, it is \emph{uncrossed}. A drawing of a graph in the plane is called \emph{1-planar} if each edge is crossed
at most once. Let $G$ be a graph with a drawing $D$.
Let $H$ be a subgraph of $G$. The subdrawing $\left.D\right|_H$ of $H$ induced by $D$ is called a \emph{restricted drawing }of $D$.

\section{Several technical lemmas}


In this section we give some lemmas. The first one is the well-known Tutte-Berge Formula, which gives a formula on the size of a maximum matching of a graph. 

\begin{lem}[Tutte-Berge \cite{MR0100850}]\label{lem:tutteberge} 
The size of a maximum matching  of a graph $G$ of order $n$ equals the minimum, over all $S \subseteq V(G)$, of $\frac{1}{2}(n-(\odd(G- S)-|S|))$.
\end{lem}

\begin{lem}[Karpov \cite{zbMATH06347737}]\label{biparite1_planar}
Let $G$ be a bipartite and 1-planar graph of order $n\ge 4$. Then $e(G)\le 3n-8$.
\end{lem} 


Let $G$ be a graph and $uv \in E(G)$. The \emph{contraction} of $uv$ in $G$ results in a new graph, denoted by $G / uv$,  in which the vertices $u$ and $v$ are identified as a single vertex, and all edges incident to either $u$ or $v$ are reassigned to this new vertex, with any resulting self-loops and multiple edges deleted.  The contraction of an edge in a graph is a standard graph operation. For a 1-planar drawing $D$ of a 1-planar graph $G$, there exists a natural geometric interpretation of contracting an uncrossed edge $uv$ in $D$, which preserves 1-planarity: The nodes $u$ and $v$, together with the line $uv$, can be regarded as a single ``supernode'', represented by the gray region on the left side of Fig.~\ref{Contraction}, whose boundary is infinitesimally close to $u$, $v$, and the line $uv$, ensuring that no other vertices or crossing points lie inside the supernode. Contracting an uncrossed edge $u v$ corresponds to merging the points representing $u$ and $v$ in the plane into a single vertex $w$, which is placed at the original position of $u$, and reattaching the edges incident with $u$ and $v$ to $w$ within the drawing. These contraction operation can be carried out within the geometric region of the supernode introduced above (shown in the right side of  Figure~\ref{Contraction}), and thus  edges incident with $u$ or $v$ remain crossed at most once, while all other edges not incident with $u$ or $v$ remain unchanged. Therefore, this geometric operation of contracting an uncrossed edge introduces no new crossings and thus preserves the 1-planarity of the drawing.

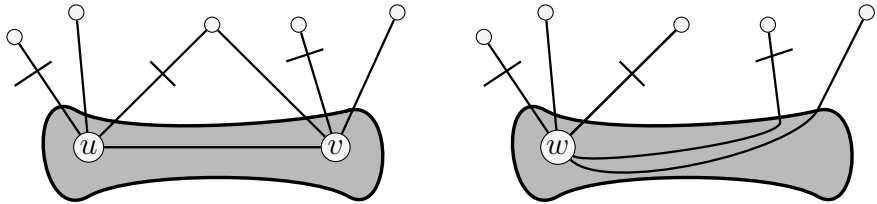
\begin{figure}[H]
\centering
\begin{tikzpicture}[scale=0.65]
	\begin{pgfonlayer}{nodelayer}
		\node [style=whitenode] (0) at (-5, 0) {$u$};
		\node [style=whitenode] (1) at (0, 0) {$v$};
		\node [style=none] (2) at (-5.25, 0.75) {};
		\node [style=none] (3) at (-5, -1) {};
		\node [style=none] (4) at (0.25, 0.75) {};
		\node [style=none] (5) at (0.25, -1) {};
		\node [style=whitenode] (6) at (-6.5, 2.25) {};
		\node [style=whitenode] (7) at (-5.25, 2.75) {};
		\node [style=whitenode] (8) at (-0.75, 2.5) {};
		\node [style=whitenode] (9) at (1.25, 2.75) {};
		\node [style=none] (10) at (-1, 1.75) {};
		\node [style=none] (11) at (-0.25, 2) {};
		\node [style=whitenode] (12) at (-2.5, 2.5) {};
		\node [style=none] (13) at (-3.75, 1.75) {};
		\node [style=none] (14) at (-3.25, 1.25) {};
		\node [style=whitenode] (15) at (4.5, 0) {$u$};
		\node [style=whitenode] (16) at (4.5, 0) {$w$};
		\node [style=none] (17) at (4.25, 0.75) {};
		\node [style=none] (18) at (4.5, -1) {};
		\node [style=none] (19) at (9.75, 0.75) {};
		\node [style=none] (20) at (9.75, -1) {};
		\node [style=whitenode] (21) at (3, 2.25) {};
		\node [style=whitenode] (22) at (4.25, 2.75) {};
		\node [style=whitenode] (23) at (8.75, 2.5) {};
		\node [style=whitenode] (24) at (10.75, 2.75) {};
		\node [style=none] (25) at (8.5, 1.75) {};
		\node [style=none] (26) at (9.25, 2) {};
		\node [style=whitenode] (27) at (7, 2.5) {};
		\node [style=none] (28) at (5.75, 1.75) {};
		\node [style=none] (29) at (6.25, 1.25) {};
		\node [style=none] (30) at (3, 1.25) {};
		\node [style=none] (31) at (3.75, 1.75) {};
		\node [style=none] (33) at (9, 0.5) {};
		\node [style=none] (34) at (-6.5, 1.25) {};
		\node [style=none] (35) at (-5.75, 1.75) {};
	\end{pgfonlayer}
	\begin{pgfonlayer}{edgelayer}
		\draw [style={shadow_silver}] (4.center)
			 to [in=-30, out=15, looseness=1.50] (5.center)
			 to [in=390, out=150, looseness=0.50] (3.center)
			 to [in=150, out=-150, looseness=1.75] (2.center)
			 to [in=195, out=-30, looseness=0.50] cycle;
		\draw [style=blackedge] (0) to (1);
		\draw [style=blackedge] (6) to (0);
		\draw [style=blackedge] (0) to (7);
		\draw [style=blackedge] (8) to (1);
		\draw [style=blackedge] (1) to (9);
		\draw [style=blackedge] (10.center) to (11.center);
		\draw [style=blackedge] (12) to (0);
		\draw [style=blackedge] (12) to (1);
		\draw [style=blackedge] (13.center) to (14.center);
		\draw [style={shadow_silver}] (19.center)
			 to [in=-30, out=15, looseness=1.50] (20.center)
			 to [in=390, out=150, looseness=0.50] (18.center)
			 to [in=150, out=-150, looseness=1.75] (17.center)
			 to [in=195, out=-30, looseness=0.50] cycle;
		\draw [style=blackedge] (15) to (16);
		\draw [style=blackedge] (21) to (15);
		\draw [style=blackedge] (15) to (22);
		\draw [style=blackedge] (25.center) to (26.center);
		\draw [style=blackedge] (27) to (15);
		\draw [style=blackedge] (27) to (16);
		\draw [style=blackedge] (28.center) to (29.center);
		\draw [style=blackedge] (31.center) to (30.center);
		\draw [style=blackedge] (24) to (19.center);
		\draw [style=blackedge, in=-45, out=-120, looseness=0.50] (19.center) to (16);
		\draw [style=blackedge] (23) to (33.center);
		\draw [style=blackedge, in=-30, out=-105, looseness=0.25] (33.center) to (16);
		\draw [style=blackedge] (35.center) to (34.center);
	\end{pgfonlayer}
\end{tikzpicture}

\caption{Contraction of an uncrossed edge $uv$ inside a supernode (gray region)}
\label{Contraction}
\end{figure}

The following lemma follows readily from the above analysis (note that it may not hold for $e$ being a crossed edge).
\begin{lem}\label{keep1_planar}
Let $G$ be a $1$-plane graph. If $e$ is an uncrossed edge, then  $G / e$ is also $1$-plane.
\end{lem}

A 1-planar drawing 
$D$ is \emph{nice} if for every pair of crossing edges in  $D$ all of their associated edges are uncrossed.

\begin{remark}\label{remark:kite}
Let $G$ be a graph with a 1-planar drawing $D$.  If $D$ is nice, then 
for any subgraph $H$  of $G$, $D|_H$ is nice. 
\end{remark}

\begin{lem}\label{lem:4_connected_kite}
Let $G$ be a graph with a 1-planar drawing $D$.  If $G$ is 4-connected, then $D$ is nice.
\end{lem}
\begin{proof}
Suppose $D$ is not nice. Then there exists a pair of edges $ab$ and $cd$ that cross at $\alpha$, with one of their associated edges being crossed.  Without loss of generality, assume that $ac$ crosses $xy$, with $x$ and $y$ lying inside and outside the region bounded by the edge $ac$ and the half-edges $c\alpha$ and $\alpha a$, respectively, as shown in Figure~\ref{fig:nice}. By 1-planarity, it then easily follows that $\{a, c, y\}$ forms a vertex-cut of $G$, a contradiction to the premise that  $\kappa(G)\ge 4$.
\end{proof}

\begin{figure}[H]
\centering
\begin{tikzpicture}[scale=0.45]
	\begin{pgfonlayer}{nodelayer}
		\node [style=whitenode] (0) at (-3, 3) {$a$};
		\node [style=whitenode] (1) at (3, 3) {$d$};
		\node [style=whitenode] (2) at (-3, -3) {$c$};
		\node [style=whitenode] (3) at (3, -3) {$b$};
		\node [style=whitenode] (4) at (-2, 0) {$x$};
		\node [style=whitenode] (5) at (-5, 0) {$y$};
		\node [style=none] (6) at (0, -1) {$\alpha$};
	\end{pgfonlayer}
	\begin{pgfonlayer}{edgelayer}
		\draw [style=blackedge] (0) to (3);
		\draw [style=blackedge] (1) to (2);
		\draw [style=blackedge] (4) to (5);
		\draw [style=blackedge] (0) to (2);
	\end{pgfonlayer}
\end{tikzpicture}

\caption{A vertex-cut $\{a,c,y\}$  arising by the contradiction assumption}
\label{fig:nice}
\end{figure}

In the following, we provide some  properties of type-2A 1-planar graphs. The first lemma is straightforward.

\begin{lem}\label{lem:delete1_planar}
Let $G$ be a  graph with a type-2A 1-planar drawing $D$.  Then the following statements hold. 
\begin{itemize}
  \item[(i)]If $H$ is an induced subgraph of $G$, then $D|_H$ is  type-2A 1-planar, and
  \item[(ii)]if $D$ is nice and $X$ is the set of all crossed edges in $D$, and   
$H$ is obtained from $G$ by deleting some edges from $X$, then $D|_H$ is  type-2A 1-planar.
\end{itemize}

\end{lem}
\begin{proof}
 Clearly (i) holds. Since $D$ is nice, any crossed edge does not serve as an associated edge for a crossing pair in $D$. Thus by removing only the crossed edges, $D|_H$ is still  type-2A 1-planar, as desired.
\end{proof}

The following lemma follows directly from the definition of type-2A 1-planar drawings.

\begin{lem}\label{lem:localcrossing}
Let $G$ be a graph with a type-2A 1-planar drawing $D$. If $x_1x_3$ crosses $x_2x_4$ in $D$, then for any $i$ where $1\le i\le 4$, $x_i$ is adjacent to $x_{i-1}$ or $x_{i+1}$ where $x_0=x_4$ and $x_5=x_1$.
\end{lem}

\begin{lem}\label{lem:componentsuncrossed}
Let $G$ be a graph with a type-2A 1-planar drawing $D$, and let $S$ be a vertex-cut of $G$. Then the following statements hold.

\begin{itemize}
  \item [(i)] For any two distinct components $F$ and $F'$ of $G-S$, if $ab \in E(F)$ and $cd \in E_G(S, V(F'))$, then the edges $ab$ and $cd$ do not cross in  $D$.
  \item [(ii)] Any two edges from distinct components of  $G-S$ do not cross each other in $D$. 
\end{itemize}

\end{lem}
\begin{proof}

(i). Suppose that $ab$ crosses $cd$ in $D$. Without loss of generality, we may assume that $c \in V(S)$ and $d\in V(F')$.   Since $D$ is type-2A, we have $ad$ or $ bd \in E_G(V(F),V(F')) $, a contradiction to  that $F$ and $F'$ are distinct components.

(ii). Suppose that for two components $F$ and $F^{\prime}$ of $G-S$, an edge from $F$ crosses an edge from $F^{\prime}$. Since $D$ is type-2A, there must exist an associated edge connecting $F$ and $F^{\prime}$, which contradicts the assumption that $F$ and $F^{\prime}$ are distinct components.
\end{proof}

The following proposition follows directly from Lemma \ref{lem:componentsuncrossed} and describes, in a type-2A 1-planar graph $G$, which edges cross a given edge from a  component of $G-S$ where $S$ is a vertex-cut of $G$.
\begin{prop}\label{lem:crosing}
Let $G$ be a graph with a type-2A 1-planar drawing $D$. Let $S$ be a vertex-cut of $G$, and $F$ be a component of $G - S$. Let $ab$ be an edge of $F$. If $ab$ is crossed by an edge $xy$ in $D$, then $xy$ can only belong to one of the following three cases.
\begin{itemize}
  \item [(a)]  $x,y \in S$;
  \item [(b)]  $x,y \in V(F)$; and
 \item [(c)] $x\in S$ and $y\in V(F)$, or  $y\in S$ and $x\in V(F)$.
\end{itemize}
\end{prop}

\begin{lem}\label{lem:type3_preserves_connected}
Let $G$ be a type-2A 1-plane graph. Let $G'$ be the graph obtained from $G$ by removing one edge from each pair of crossing edges. If $G$ is connected and nice, then $G'$ is a connected, spanning and planar subgraph of $G$.
\end{lem}

\begin{proof}
 Clearly,  $G'$ is spanning and  planar.  Suppose that $G'$ is not connected.  Since $G$ is connected, there  exist two distinct components $F^1$ and $F^2$ of $G'$ such that  $|E_G(V(F^1),V(F^2))| \ge 1$. Without loss of generality, assume that $ab \in E_G(V(F^1),V(F^2))$, where $a \in V(F^1)$ and $b \in V(F^2)$. Clearly, the edge $a b$ is crossed in $G$. We assume that $a b$ crosses with $c d$. Clearly, $c d$ is an edge of $G^{\prime}$, and thus it is impossible for $c$ and $d$ to be such that one is in $F^1$ and the other in $F^2$. Thus, we only need to consider two cases.

\begin{case} 
$c \in V(F^1)\cup V(F^2) \cup V(F)$  and $d \in V(F)$, where $F$ is a component of $G^{\prime}$ distinct from $F^1$ and $F^2$.
\end{case} 

For $c \in V(F^1)$, since $G$ is type-2A, by Lemma \ref{lem:localcrossing},  either $b c \in G^{\prime}$ or $b d \in G^{\prime}$.  The former ($b c \in G'$) contradicts that $F^1$ and $F^2$ are  distinct components, while the latter ($b d \in G'$) contradicts that $F$ and $F^2$ are  distinct components.  
Similarly, for  $c \in V(F^2)$ or $c\in V(F)$, a contradiction also arises.

\begin{case} 
 $c,d \in V(F^1)$ or  $c,d \in V(F^2)$ .  
\end{case} 
For $c,d \in V(F^1)$, since $G$ is type-2A,  by Lemma \ref{lem:localcrossing},  either $b c \in E(G')$ or $b d \in E(G')$, which contradicts that $F^1$ and $F^2$ are  distinct components.   Symmetrically, for $c,d \in V(F^2)$, a contradiction also arises. 

As the above two cases lead to contradictions, it follows that $G'$ is connected, thereby establishing the proposition.
\end{proof}

%



\begin{lem}\label{lem:asso}
Let $G$ be a graph with a type-2A 1-planar drawing $D$. Let $X$ and $Y$ be two distinct vertex subsets of $G$. Let $G'$ be the graph obtained by removing all crossed edges in $E_G(X,Y)$ that cross some edge with both endvertices in $X$. If $D$ is nice, then 
\[
N_{G'}(X,Y) = N_{G}(X,Y).
\]
\end{lem}

\begin{proof}
It suffices to show that for each deleted crossed edge, its endvertex in $Y$ is still preserved.
We select any edge $xy$ from $E(G)\setminus E(G')$. Without loss of generality, let $xy$ cross $ab$, where $x \in X$, $y \in Y$, and $a,b \in X$.  
Moreover, since $D$ is type-2A 1-planar, by Lemma \ref{lem:localcrossing} we have $ya \in E(G)$ or $yb \in E(G)$. As $D$ is nice, whichever of $ya$ or $yb$ exists is uncrossed in $D$. Hence $ya \in E(G-xy)$ or $yb \in E(G-xy)$. Therefore, $y \in N_{G'}(X,Y)$. Clearly, by   a similar argument, deleting the remaining edges in $E(G) \backslash E\left(G^{\prime}\right)$ one by one, we obtain $N_{G^{\prime}}(X, Y)=N_G(X, Y)$.
\end{proof}

\section{Proof of Theorem \ref{thm:1}}

 The proof of the existence of near-perfect matchings of 6-connected type-2A 1-planar graphs $G$ relies on the Tutte-Berge formula. However, we actually bound  $\max_{S\subseteq V(G)} c(G-S)-S$  where $c(G - S) \geq \operatorname{odd}(G - S)$ (i.e., counting components of $G-S$ with both even and odd cardinality; see details in the following Proposition \ref{prop:main}). This bound (when restricted to sets $S$ with $c(G - S) \geq 2$ ) is also known as the \emph{scattering number} $s(G):=\max _{S \subseteq V(G), c(G - S) \geq 2}\{c(G - S)-|S|\}$ \cite{MR0491356}. The proof is based on a reduction scheme that removes certain crossed edges and eventually yields a particular 1-planar minor, and this idea has appeared in recent work by Biedl and Wittnebel \cite{MR4813988}; however, there are many differences in the details.  A key technical step exploits properties of type-2A 1-planar graphs: by deleting certain edges, each component of $G-S$ becomes a planar connected subgraph, which can then be contracted into a single vertex. We then leverage 6-connectivity and bipartite-edge-counting to obtain a  upper bound.
 
\begin{prop}\label{prop:main}
Let $G$ be a  type-2A 1-planar graph. If $G$ is 6-connected, then for any subset $S \subseteq V(G)$,  the inequality $c(G - S) -|S|\le 1$ holds.
\end{prop}
\begin{proof}
Let $D$ be a type-2A 1-planar  drawing of $G$. If $S$ is not a vertex-cut, then we have
$
 c(G-S) \leq 1
$.
In this case, the desired inequality clearly holds. We now  assume that $S$ is a vertex-cut. Clearly, $|S|\ge 6$. We may assume that $F^1, \ldots, F^\ell$ are the components of $G-S$.
For each $i$ where $1\le i\le \ell$, we proceed with the following two operations to delete edges from $G$. (The two operations are independent and order does not matter.)

\begin{itemize}
 \item[I.]  For every edge $uv \in E(F^i)$, if it crosses an edge $xy$ in $D$ with both endvertices in $S$, or with one endvertex in $F^i$ and the other in $S$, then $xy$ is deleted.
  \item[II.] For each pair of crossed edges in $D|_{F^i}$,  we delete one of them.
\end{itemize}

The following Figure \ref{fig:op} shows a local illustration. Vertices $s_1,s_2,\dots,s_6$ belong to the set $S$, while $a,b,c,d,e,h$ belong to one component of $G-S$.
Under Operation I, we delete the edges $s_4s_6$ and $s_3c$; under Operation II, we delete one of the edges $ac$ (or $bd$).

\begin{figure}[H]
\centering
\begin{tikzpicture}[scale=0.75]
\tikzset{
  pinknode/.style={circle, draw=black, fill=pink!60, minimum size=5mm, inner sep=0pt}
}

	\begin{pgfonlayer}{nodelayer}
		\node [style=pinknode] (0) at (-8, 2) {$d$};
		\node [style=pinknode] (1) at (-5, 2) {$c$};
		\node [style=pinknode] (2) at (-8, 0) {$a$};
		\node [style=pinknode] (3) at (-6.5, 0) {$b$};
		\node [style=pinknode] (4) at (-4, 0) {$e$};
		\node [style=whitenode] (5) at (-5, -3) {$s_3$};
		\node [style=pinknode] (6) at (-2.5, 0) {$h$};
		\node [style=whitenode] (7) at (-4, -3) {$s_4$};
		\node [style=whitenode] (8) at (-1, -3) {$s_6$};
		\node [style=whitenode] (10) at (-7, -3) {$s_2$};
		\node [style=whitenode] (11) at (-8, -3) {$s_1$};
		\node [style=whitenode] (12) at (-3, -3) {$s_5$};
		\node [style=pinknode] (13) at (1.75, 2) {$d$};
		\node [style=pinknode] (14) at (4.75, 2) {$c$};
		\node [style=pinknode] (15) at (1.75, 0) {$a$};
		\node [style=pinknode] (16) at (3.25, 0) {$b$};
		\node [style=pinknode] (17) at (5.75, 0) {$e$};
		\node [style=whitenode] (18) at (4.75, -3) {$s_3$};
		\node [style=pinknode] (19) at (7.25, 0) {$h$};
		\node [style=whitenode] (20) at (5.75, -3) {$s_4$};
		\node [style=whitenode] (21) at (8.75, -3) {$s_6$};
		\node [style=whitenode] (22) at (2.75, -3) {$s_2$};
		\node [style=whitenode] (23) at (1.75, -3) {$s_1$};
		\node [style=whitenode] (24) at (6.75, -3) {$s_5$};
		\node [style=none] (25) at (0, 0) {$\rightarrow$};
	\end{pgfonlayer}
	\begin{pgfonlayer}{edgelayer}
		\draw (0) to (3);
		\draw (1) to (2);
		\draw (2) to (3);
		\draw (0) to (1);
		\draw (3) to (4);
		\draw (1) to (5);
		\draw (4) to (6);
		\draw [bend left=75, looseness=4.00] (7) to (8);
		\draw (1) to (4);
		\draw (2) to (10);
		\draw (10) to (3);
		\draw (5) to (4);
		\draw (6) to (8);
		\draw (4) to (7);
		\draw (3) to (5);
		\draw (2) to (11);
		\draw (11) to (3);
		\draw (6) to (12);
		\draw (11) to (10);
		\draw (10) to (5);
		\draw (7) to (5);
		\draw (7) to (12);
		\draw (12) to (8);
		\draw (13) to (16);
		\draw (15) to (16);
		\draw (13) to (14);
		\draw (16) to (17);
		\draw (17) to (19);
		\draw (14) to (17);
		\draw (15) to (22);
		\draw (22) to (16);
		\draw (18) to (17);
		\draw (19) to (21);
		\draw (17) to (20);
		\draw (16) to (18);
		\draw (15) to (23);
		\draw (23) to (16);
		\draw (19) to (24);
		\draw (23) to (22);
		\draw (22) to (18);
		\draw (20) to (18);
		\draw (20) to (24);
		\draw (24) to (21);
	\end{pgfonlayer}
\end{tikzpicture}

\caption{Illustration of Operations I and II.}
\label{fig:op}
\end{figure}
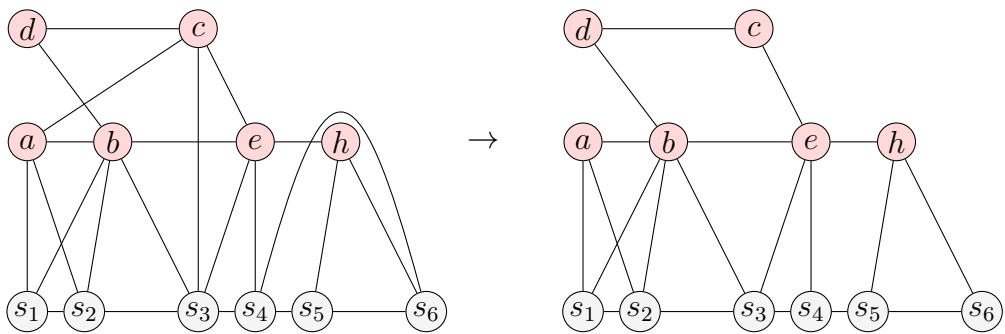

After performing the above two operations, we obtain a spanning subgraph of $G$, denoted by $H$, and let $\varphi=D|_{H}$. 
In the above process, only crossed edges are removed, so by Remark \ref{remark:kite}, $\varphi$ is  nice. Clearly, $\varphi$ is type-2A 1-planar.  Furthermore, we have the following claim concerning the structure of $H$.

\begin{claim}\label{claim:H}
For any $i$ with $1\le i \le \ell$, the following statements on $H$ hold.
\begin{itemize}
\item[(i)]  The vertices of  $F^i$ still induce a  component, namely $F^i_*$, of $H$;
\item[(ii)]  every edge in $F^i_*$ is uncrossed in $\varphi$; and
\item[(iii)] $N_{H}(V(F^i_*)) \subseteq S$ and $\left|N_{H}(V(F^i_*),S)\right| \geq 6$.
\end{itemize}

\end{claim}

\begin{proof}(i). Let $H'$ be a 1-plane graph obtained by completing the operation I from $G$. No  edges (and vertices) of $F^i$ are removed, so it remains connected in $H'$.  By Lemma \ref{lem:delete1_planar} and Remark \ref{remark:kite}, $H'$ is a nice type-2A 1-plane graph. Note that Operation II does not affect any structure of $H[V(F^i)]$. By Lemma \ref{lem:delete1_planar} and Remark \ref{remark:kite}, $H'[V(F^i)]$ is also type-2A 1-plane and nice.
 Furthermore, by Lemma \ref{lem:type3_preserves_connected}, (i) holds. 

(ii). By  Proposition \ref{lem:crosing}, in $G$, for any $i$ where $1 \le i \le l$, if an edge $ab \in E(F^i)$ crosses an edge $cd$, then $\{c,d\} \subseteq S\cup V(F^i)$. Hence, after completing Operations I and II, every edge of $F^i_*$ is uncrossed in $\varphi$, and statement (ii) holds
 
 (iii). Since $\kappa(G)\ge 6$, it is easy to see that $N_{G}(V(F^i)) \subseteq S$ and $\left|N_G(V(F^i),S)\right| \geq 6$.  Applying Lemma \ref{lem:asso} with $H = G'$, $V(F^i) = X$, and $S = Y$ yields the desired statement (iii).
\end{proof}

Next, we delete all edges in $H[S]$.  By Claim \ref{claim:H}(i), for any $i(1 \leq i \leq \ell), F_*^i$ contains a spanning tree, and we successively contract the edges along this tree, ultimately reducing $F_*^i$ to a single vertex $f^i$.  
 Clearly, the resulting graph, namely $B$, is a bipartite graph with partite sets $S$ and $T$, where $T := \{f^1, f^2, \dots, f^l\}$.
 Note that the above edge contraction may cause two adjacent edges to cross with each other in the resulting  drawing of $B$. If this phenomenon appears, we can modify the drawing so that this two adjacent edges are no longer crossed.

\begin{claim}\label{claim:B}
The following properties on $B$ hold.

\begin{itemize}
\item[(i)]  $B$ is 1-planar; and
\item[(ii)] $\left|N_B(f^i,S)\right| \geq 6$ where $1\le i\le \ell$.
\end{itemize}
\end{claim}

\begin{proof}
Firstly, (i) follows from  Lemma \ref{keep1_planar}. By the definition of edge contractions, we have $N_B\left(f^i, S\right)=N_H\left(F_*^i, S\right)$. Combining this with Claim \ref{claim:H} (iii), we have $\left|N_B\left(f^i, S\right)\right| \geq 6$. Hence, (ii) holds.
\end{proof}

Now we ﬁnish the proof of the proposition.
By Claim \ref{claim:B}(ii), $6l\le e(B)$. Clearly $n(B)\ge 7>4$. Combining Claim \ref{claim:B}(i) with Lemma \ref{biparite1_planar}, $e(B) \le 3(l+|S|)-8$. Thus $l-|S|\le -\frac{8}{3}$, that is $c(G-S)-|S|\le -\frac{8}{3}<1$.
 This completes the proof.
\end{proof}

\subsection{Proof of Theorem \ref{thm:1}} 

For any vertex subset $S$ of $G$, it is clear that $\odd(G-S) \leq c(G-S)$. Combining this with Proposition \ref{prop:main}, we have $\odd(G-S)-|S| \leq1$. Hence, by Lemma \ref{lem:tutteberge}, we  have $\alpha^{\prime}(G)=\lfloor n(G) / 2\rfloor$, as desired.

\section{Conclusion}
In this paper, we show that every 6-connected type-2A 1-planar graph admits a near-perfect matching. However, for all type-2, type-1, and even more general 1-planar graphs, the problem remains unsolved. It is worth noting that there do exist 6-connected 1-planar graphs that are not of type-2A; an explicit example can be found in Figure~5 of Noguchi, Ota, and Suzuki \cite{MR4958078}.
In particular, the approach used in Theorem~\ref{thm:1}—contracting each component into a single vertex while maintaining the 1-planarity of the resulting bipartite graph—does not immediately generalize to resolve the  problem. Therefore, a more detailed structure and novel methods are required. Finally, the known results and open problems are summarized in the following table.
\begin{table}[H]
\footnotesize
\centering
\begin{tabular}{|c|c|c|p{3.5cm}|c|c|}

\hline

\textbf{} & \textbf{Type-4 1-planar} & \textbf{Type-3 1-planar} & \textbf{Type-2 1-planar} & \textbf{Type-1 1-planar} & \textbf{1-planar} \\
\hline
3 & \no \makebox[0pt][l]{(i)} & \no &  \no & \no  & \no \\
\hline
4 & \yes \makebox[0pt][l]{(i)} & \no & \no &\no & \no \\
\hline
5 & \yes & \no \makebox[0pt][l]{(ii)} & \no & \no & \no \\
\hline
6 & \yes & \yes \makebox[0pt][l]{(ii)} & \textbf{?} (\yes for type-2A)  \makebox[0pt][l]{(ii)}& \textbf{?} & \textbf{?}  \\
\hline
7 & \yes  &\ding{51}  & \textbf{?} (\yes for type-2A) & \textbf{?}  & \textbf{?} \\
\hline
\end{tabular}

\medskip

\begin{tabular}{llcl}
i. & Fabrici et al. \cite{MR4130388} & ii. & This paper\\

\end{tabular}
\caption{Near-perfect matchings in 1-planar graphs and related families, categorized by connectivity from 3 to 7. Entries marked with \ding{51} indicate that all such graphs having  near perfect matchings; \ding{55} denote the existence of examples having no near perfect matchings; the symbol \textbf{?} is used to denote open problems.} 
\label{tab:type-connectivity}
\end{table}

\section*{Acknowledgements}
The first author is grateful to Prof. Therese Biedl for giving a proof of the existence of near-perfect matchings in 4-connected planar graphs that does not rely on Tutte’s Hamiltonian cycle theorem on planar graphs.
The authors declare that they have no conflicts of interest. This research was supported by grant from the National Natural Science Foundation of China (Grant Nos. 12271157, 12371346) and the Postdoctoral Science Foundation of China (Grant No. 2024M760867).

%

\end{document}